\documentclass[10pt,a4paper]{article}
\usepackage{setspace}

\usepackage{amsfonts,amsmath,amssymb}
\usepackage{graphicx}
\usepackage{sectsty}
\usepackage{fullpage}
\newtheorem{theorem}{Theorem}[section]
\newtheorem{lemma}[theorem]{Lemma}

\newtheorem{proposition}[theorem]{Proposition}

\newenvironment{proof}[1][Proof]{\begin{trivlist}
\item[\hskip \labelsep {\bfseries #1}]}{\end{trivlist}}
\numberwithin{equation}{section}

\sectionfont{\large}

\begin{document}
\title{A Super-Exponential Decaying Property of Odd-Dimensional Wave Scattered by an Obstacle}
\author{Lung-Hui Chen$^1$}\maketitle\footnotetext[1]{Department of
Mathematics, National Chung Cheng University, 168 University Rd.,
Min-Hsiung, Chia-Yi County 621, Taiwan. Email:
lhchen@math.ccu.edu.tw; mr.lunghuichen@gmail.com. Fax: 886-5-2720497. The author is supported by NSC Grant
97-2115-M194-010-MY2.}
\begin{abstract}
We examine an inverse backscattering property of wave motion
imposed by an obstacle. We show that if the wave propagator
 decays super-exponentially along the back-scattered geodesics, then the involved scatterer
must be trivial. In particular, if the fundamental solution decays
super-exponentially some time after $t=0$, it vanishes for all
time. We use finite speed of propagation in this article.
\end{abstract}
\section{Introduction and main results}
Let $\Gamma$ be an embedded hypersurface in $\mathbb{R}^n$ with odd $n\geq3$ such that
\begin{equation}
\mathbb{R}^n\setminus \Gamma=\Omega\cup \mathcal{O},\mbox{ with
}\overline{\mathcal{O}}\mbox{ compact and }\overline{\Omega}\mbox{
connected },
\end{equation}where both $\mathcal{O}$ and $\Omega$ are open. We
denote $\mathcal{O}$ as an obstacle and $\Omega$ as its exterior.

\par
We consider the following exterior problem. Let $u\in
\mathcal{H}^2(\Omega)$ be the solution of
\begin{eqnarray}\label{1.2}
    -\Delta u+qu=0 \mbox{ in }\Omega
\end{eqnarray}
with the boundary operator $\gamma$ on $\Gamma$ given either by
\begin{equation}\label{1.3}
\gamma u=(\partial/\partial\nu)u\mbox{ or }
\gamma u=u.
\end{equation}
We assume that $q$ is real-valued function uniformly H\"{o}lder
continuous on $\Omega\cup\Gamma$ with a compact support. Let us
use $H$ be the self-adjoint operator in $\mathcal{L}^2(\Omega)$
given by $-\Delta+q$ acting on the functions $g\in
\mathcal{H}^2(\Omega)$ with boundary condition $\gamma g=0$; $H_0$
be the self-adjoint operator in $\mathcal{L}^2(\mathbb{R}^n)$
given by $-\Delta$ acting on $\mathcal{H}^2(\mathbb{R}^n)$. We
refer to Shenk and Thoe \cite{Shenk,Shenk2,Shenk3,Shenk4} for a
review on a spectral theory and scattering theory as in the
setting~(\ref{1.2}) and~(\ref{1.3}).

\par
In this paper, we use $A(x,D_x)):=H$ outside $\mathcal{O}$ satisfying
boundary condition~(\ref{1.3}) to review Irvii's theory
\cite{Ivrii,Ivrii2}. Let us investigate the relation of the
following three PDEs in this paper. Let $u(x,y,t)$ be the solution
of
\begin{equation}
\left\{%
\begin{array}{ll}
    P(x,D_x,D_t)u(x,y,t):=(D^2_t-A(x,D_x))u(x,y,t)=0;\\
    u(x,y,t)|_{t=0}=0,\hspace{2pt}  u_t(x,y,t)|_{t=0}=\delta(x-y).
    \label{u}
\end{array}%
\right.
\end{equation}
with boundary condition~(\ref{1.3}); let $u_0(x,y,t)$ be the
solution of
\begin{equation}
\left\{%
\begin{array}{ll}
    P(x,D_x,D_t)u_0(x,y,t)=0;\\
    u_0(x,y,t)|_{t=0}=0,\hspace{2pt}
    u_{0t}(x,y,t)|_{t=0}=\delta(x-y);
    \label{u0}
\end{array}%
\right.
\end{equation}
let $u_1(x,y,t)$ be the solution of
\begin{equation}
\left\{%
\begin{array}{ll}
    P(x,D_x,D_t)u_1(x,y,t)=0;\\
    u_1(x,y,t)|_{t=0}=u_{1t}(x,y,t)|_{t=0}=0,\hspace{2pt}B_{\pm}u_1(x,y,t)=-B_{\pm}u_0(x,y,t),
    \label{u1}
\end{array}%
\right.
\end{equation}
where $B_{-}v=v|_\Gamma$ and $B_+v=\frac{\partial}{\partial
n}v|_\Gamma$. In operator form, $u(x,y,t)$ is the Schwartz kernel of
$\frac{\sin t \sqrt{H}}{\sqrt{H}}$. We see that
$u_1(x,y,t)=u(x,y,t)-u_0(x,y,t)$.

\par
We have to remind ourselves of the relation of~(\ref{u0})
and~(\ref{u1}) to the following Cauchy problem for the wave
equation: let $U(x,y,t)$ be the solution of
\begin{equation}\label{U}
\left\{%
\begin{array}{ll}
    P(x,D_x,D_t)U(x,y,t):=(D^2_t-A(x,D_x))U(x,y,t)=0;\\
    U(x,y,t)|_{t=0}=\delta(x-y),\hspace{2pt}
    U_{t}(x,y,t)|_{t=0}=0,

\end{array}%
\right.
\end{equation}
with either boundary condition~(\ref{1.3}). Let
$U_0(x,y,t)$ be the solution of
\begin{equation}\label{U_0}
\left\{%
\begin{array}{ll}
    P(x,D_x,D_t)U_0(x,y,t)=0;\\
    U_0(x,y,t)|_{t=0}=\delta(x-y),\hspace{2pt}
    U_{0t}(x,y,t)|_{t=0}=0.

\end{array}%
\right.
\end{equation}
We define
\begin{equation}
U_1(t,x,y):=U(t,x,y)-U_0(t,x,y).\label{U_1}
\end{equation}
Formally, we write $U(t)$ or $\cos t\sqrt{H}$ as the solution
operator. In terms of Fourier transform, $\cos t\sqrt{H}$ and
$\frac{\sin t \sqrt{H}}{\sqrt{H}}$ differs by $|\lambda|$.
$\lambda$ is the frequency variable. In short time, the wave trace
is supported in a region containing $\Omega$. In that case, the
support can be embedded into a torus. We remove the set
$\mathcal{O}$ from this torus. That is the manifold the analysis
in this paper is carried out. See section 2 for a picture.

\par
In S\'{a} Barreto and Zworski \cite{SaBarreto and
Zworski,SaBarreto and Zworski2}, they try to answer the following
question: let $E(t,x,y)$ be the fundamental solution of the
following perturbed wave equation
\begin{eqnarray}\left\{%
\begin{array}{ll}
    (D^2_t-P)E(t,x,y)=0; &  \\
    E(0,x,y)=0; &  \\
    \partial_tE(0,x,y)=\delta(x-y), &  \\
\end{array}%
\right.\label{schrodinger}
\end{eqnarray}
where $P$ is an elliptic self-adjoint operator in
$L^2(\mathbb{R}^n,\sqrt{\bar{g}})$ defined as
\begin{eqnarray}
P:=-\frac{1}{\sqrt{\bar{g}}}\sum_{i,j=1}^n\partial_{x_i}\sqrt{\bar{g}}g^{i,j}\partial_{x_j}+V,
\hspace{2pt}\bar{g}=\det(g_{ij}),\hspace{2pt}(g^{ij})=(g_{ij})^{-1},
\end{eqnarray}where $V$, $g_{ij}$ are among smooth functions with
bounded derivatives, $\mathcal{C}^\infty_b(\mathbb{R}^n)$, such
that $V(x)$ and $|g_{ij}-\delta_{ij}|$ decays super-exponentially.
Does a super-exponentially decaying fundamental solution
$E(t,x,y)$ of~(\ref{schrodinger}) inside its characteristic cone
implies the solution $E(t,x,y)$ actually vanishes there? They give
affirmative answer there. Is there a similar property valid for an
obstacle scatterer?

\par
On the other hand, the relation between the wave decaying speed and the location of the
poles from the meromorphic continuation of the Green's function in $\mathbb{C}$ is classical. For
instance, let $u(x,t)$ be the solution of the Cauchy problem:
\begin{eqnarray}
\left\{%
\begin{array}{ll}
    D_t^2u(x,t)+Hu(x,t)=\varphi(x,t); \\
    u(x,0)=f(x); \\
    D_t(x,0)=g(x), \\
\end{array}%
\right.
\end{eqnarray}
where $H:=-\Delta+q$ and the regularity condition on $q,\varphi,f,g$ are specified as in
Thoe \cite{Thoe}. Suppose that $H$ has neither $L^2$-discrete spectrum nor $L^2$-embedded spectrum at
$0$. Then, it is shown in Thoe \cite[sec.3]{Thoe} that the solution
$u(x,t)$ behaves like $O(e^{-\gamma t})$, $\gamma$, as
$t\rightarrow\infty$, in such a fashion that $\gamma$ is any
positive number less than the minimal distance of the poles of its
Green's function from the real axis in $\mathbb{C}$. Is the
statement valid for an obstacle scatterer?
\par
We state the main result in this paper as
\begin{theorem}
Let $U_1(t,x,y)$  be described as in~(\ref{U_1}). If for all $x\in
\mathbb{R}^n$, $n\geq 3$, odd, and for all $N\in\mathbb{N}$ and
for some constants $T,C$ such that for all $|t|>T>0$,
\begin{equation}\label{1.13}
|U_1(t,x,x)|\leq Ce^{-N|t|},
\end{equation}then $\mathcal{O}=\phi$.
\end{theorem}
Comparing the assumption~(\ref{1.13}) with \cite{SaBarreto and
Zworski,SaBarreto and Zworski2}, the constant $C$ in~(\ref{1.13})
is independent of $x$. In particular, letting $N$ goes to
infinity, we have $U_1(x,x,t)\equiv 0$ for $|t|>T$. In this case,
we will show the scatterer $\mathcal{O}$ is void. We have
$U_1(x,x,t)=U(x,x,t)-U_0(x,x,t)\equiv0 $ by the uniqueness of the
wave equation~(\ref{U_0}). $U_1(x,x,t)$ is entirely zero for all
$t$. Hence, under assumption~(\ref{1.13}), finite speed of
propagation makes our analysis here behaves as a compact case.

\section{Spectral analysis}
We start with this paper with the meromorphic extension theory for
resolvent operator
$$(H-\lambda^2)^{-1}: \mathcal{L}^2(\mathbb{R}^n)\rightarrow \mathcal{L}^2(\mathbb{R}^n).$$ It is well-known in the literature.
We refer to Sj\"{o}strand and Zworski \cite{Sjostrand and
Zworski1} that
\begin{equation}
R(\lambda):=(H-\lambda^2)^{-1}:\mathcal{L}^2(\mathbb{R}^n)\rightarrow
\mathcal{H}^2(\mathbb{R}^n),\hspace{2pt}\Im\lambda>0,\hspace{2pt}\lambda^2\notin\sigma(H),
\end{equation}
meromorphically extends to
\begin{equation}
R(\lambda):\mathcal{L}^2_{comp}(\mathbb{R}^n)\rightarrow\mathcal{H}^2_{loc}(\mathbb{R}^n),\label{2.2}
\end{equation}
where for odd $n$, $\lambda$ is defined over $\mathbb{C}$; for
even $n$, $\lambda$ is defined over $\Lambda$, the logarithmic
plane. We refer Zworski for \cite{Zworski3,Zworski2} for a scattering theory.
Let $\chi_1$, $\chi_2\in\mathcal{C}^\infty_0(\mathbb{R}^n)$ such
that $\chi_2\chi_1=\chi_1$ and $\chi_1$ covers
$\bar{\mathcal{O}}$.~(\ref{2.2}) is equivalent to
\begin{equation}
\chi_2R(\lambda)\chi_1:\mathcal{L}^2(\mathbb{R}^n)\rightarrow
\mathcal{H}^2(\mathbb{R}^n),
\end{equation}as a meromorphic family of operators. Such an
extension theory does not depend on the cutoff functions $\chi_i$,
$i=1,2.$ That means we can shrink the domain of the resolvent and
enlarge the image space to make $R(\lambda)$ defined as a
meromorphic family of operators. Alternatively, we let
$\chi\in\mathcal{C}^\infty(\mathbb{R}^n;[0,1])$ be a local cutoff
function such that $\chi_i$, $i=1,2$, has supports in
supp$(\chi)$. We alternatively say
\begin{equation}
\chi R(\lambda)\chi:\mathcal{L}^2(\mathbb{R}^n)\rightarrow
\mathcal{H}^2(\mathbb{R}^n)\mbox{ meromorphically}.
\end{equation}
From \cite{Sjostrand and Zworski1}, we know $\chi R(\lambda)\chi$
in a black box perturbation can be written as
\begin{equation}\label{2.5}
\chi
R(\lambda)\chi=\chi\{Q_0(\lambda)\chi+Q_1(\lambda_0)\chi\}(I+K(\lambda,\lambda_0))^{-1}.
\end{equation}
Without any possible confusion, we still use $R(\lambda)$ as the
extended cutoff resolvent.

\par
The poles are of finite rank and their
multiplicity
\begin{equation}
m_{\lambda_0}(R):=\mbox{rank}\int_{\partial
D(\lambda_0,\epsilon)}R(\lambda)\lambda d\lambda.
\end{equation}

\par
In addition to lemmas above, we may take the resolvent kernel $R(\lambda,x,y)$ as the
Green's function of the elliptic problem $(H-\lambda^2)u=0$ with
boundary condition~(\ref{1.3}) which is
 well-known in the setting as Shenk and Thoe \cite[Theorem 5.1]{Shenk2}
and Thoe \cite{Thoe}. In particular,
\begin{lemma}\label{2.3}
Let $\Gamma=\Gamma_1\cup\Gamma_2$ where $\Gamma_1$ and $\Gamma_2$
are two disjoint subsets of $\Gamma$. Then the elliptic
problem~(\ref{1.2}) and~(\ref{1.3}) has the Green's function
$G(x,y,\lambda)$ such that $(i)$ $G(x,y,\lambda)$ is an outgoing
solution of
\begin{eqnarray}
\left\{%
\begin{array}{ll}
(-\Delta_x+q(x)-\lambda^2)G(x,y,\lambda)=0,\forall
x\in\Omega\setminus\{y\};\\ (\frac{\partial}{\partial\nu_x}-\sigma(x))G(x,y,\lambda)=0,\forall
x\in\Gamma_1;\\ G(x,y,\lambda)=0,\forall x\in\Gamma_2,
\end{array}%
\right.
\end{eqnarray}
such that
\begin{equation}
u(y)=\int_\Gamma[u(x)\frac{\partial}{\partial\nu_x}G(x,y,\lambda)-G(x,y,\lambda)\frac{\partial}{\partial\nu_x}u(x)]
dS_x+\int_\Omega G(x,y,\lambda)(-\Delta+q(x)-\lambda^2)u(x)dx,
\end{equation}
$\forall y\in\Omega$ and for all outgoing function $u\in
\mathcal{H}^2_{loc}(\overline{\Omega})\cap\mathcal{C}^{2+\alpha}(\Omega)$

$(ii)$ For fixed $y\in\Omega$, the map
\begin{equation}
G(\lambda,x,y):\mathbb{F}\setminus{\rm resonances}
\rightarrow\mathcal{C}^1(\Omega\cup\Gamma)\cap\mathcal{C}^{2+\alpha}(\Omega)\end{equation}
is continuous; for fixed $y\in\Omega$, the map
\begin{equation}G(\lambda,x,y):\mathbb{F}\rightarrow\mathcal{C}^1(\Omega\cup\Gamma)\cap\mathcal{C}^{2+\alpha}(\Omega)\end{equation}
is meromorphic.
\end{lemma}

From~(\ref{2.5}), we can define a scattering amplitude for black
box formalism.
\begin{equation}
A(\lambda):=C_n\lambda^{n-2}
\mathbb{E}^{\phi_1}(-\lambda)[\Delta_0,\chi_2]R(\lambda)[\Delta_0,\chi]^t\mathbb{E}^{\phi_2}(\lambda),\label{2.6}
\end{equation}
where
$$
\mathbb{E}(\lambda):\mathcal{L}^2(\mathbb{S}^{n-1})\rightarrow\mathcal{C}^\infty(\mathbb{R}^n);u(\omega)\mapsto
C_n\lambda^{\frac{n-1}{2}}\int_{\mathbb{S}^{n-1}}u(\omega)e^{i\lambda
x\cdot\omega}d\omega,
$$
where $C_n$ is a constant depending on $n$. $A(\lambda)$ actually
comes from the radiation pattern of
$$R(\lambda)(-[\Delta_0,\chi]e^{i\langle
\bullet,\omega\rangle}).$$

\par
Scattering theory happens mostly on the continuous spectrum. The
scattering behavior of the solution is represented by the spectral
integration over the continuous spectrum. It is
standard in spectral analysis that
\begin{eqnarray}\nonumber
U_1(t)&=&\int_{\mathbb{R}^+}e^{-i\lambda
t}\{R(\lambda)-R_0(\lambda)-R(-\lambda)+R_0(-\lambda)\}d\lambda^2\\
&&+\Pi(0)+2\sum_{\Im\lambda_j>0}\Pi(\lambda_j).\label{116}
\end{eqnarray}
The first term on the right hand side of~(\ref{1.16}) comes from
the spectral measure
$$dE_\lambda^2:=\{R(\lambda)-R_0(\lambda)-R(-\lambda)+R_0(-\lambda)\}d\lambda^2$$
integrating along the continuous spectrum, the second one from the
possible embedded eigenvalue on the continuous spectrum and the
third one is from eigenvalues. Each discrete eigenspace are finite
dimensional. Hence, there exists a canonical kernel
$\Pi(\lambda_j,x,y)$ for each $\lambda_j$. The kernel of resolvent
is understood as the Green's function described in Lemma~\ref{2.3}
extended as a $\mathcal{L}^2$-kernel.

\par
In more general setting,
$R(\lambda)-R_0(\lambda)\in\mathcal{D}'(\Omega\times\Omega)$ by Schwartz kernel theorem. Another mathematical treatment on this
Schwartz kernel of $R(\lambda)$
is to see it as a Green's operator in sense of Ivrii
\cite[(1.3)]{Ivrii} which is an oscillatory integral.
Hence, the following identity holds in
$\mathcal{D}'(\Omega\times\Omega)$,
\begin{eqnarray}\nonumber
U_1(t,x,y)&=&\int_{\mathbb{R}^+}e^{-i\lambda
t}\{R(\lambda,x,y)-R_0(\lambda,x,y)-R(-\lambda,x,y)+R_0(-\lambda,x,y)\}d\lambda^2\\
&&+\Pi(0,x,y)+2\sum_{\Im\lambda_j>0}\Pi(\lambda_j,x,y).\label{1.16}
\end{eqnarray}
$U_1(t)$ also has a distributional trace. See Zworski \cite{Zworski2}.

\par
To continue our scattering theory, we define
\begin{equation}
s(\lambda):=\det S(\lambda)
\end{equation}where $S(\lambda)$ is the relative scattering
matrix in sense of Zworski \cite{Zworski2,Zworski3}. Let us define
$\sigma(\lambda):=\frac{i}{2\pi}\log s(\lambda)$. By functional
analysis,
\begin{equation}
\sigma'(\lambda)=\frac{i}{2\pi}\frac{s'(\lambda)}{s(\lambda)}.\label{1.23}
\end{equation}
We have a Weierstrass product form for $s(\lambda)$. Let
$m_\mu(R)$ be the multiplicity of $R(\lambda)$ near $\lambda=\mu$.
We have
\begin{equation}
P(\lambda):=\prod_{\{\mu:\rm{
resonances}\}}E(\frac{\lambda}{\mu},[m])^{m_\mu(R)},
\end{equation}
where
\begin{equation}
E(z,p):=(1-z)\exp(1+\cdots+\frac{z^p}{p}).
\end{equation}
Most important of all,
\begin{equation}
s(\lambda)=e^{g(\lambda)}\frac{P(-\lambda)}{P(\lambda)}.\label{1.26}
\end{equation}
The scattering determinant grows outside its resonances/poles like
\begin{equation}|s(\lambda)|\leq Ce^{C|\lambda|^n}, \hspace{2pt}C \mbox{
constants}.\label{1.12}
\end{equation} See Vodev
\cite{Vodev} for a black box formalism.  $g(\lambda)$ is a
polynomial of order at most $n$. See Vodev \cite{Vodev}. In particular,
\begin{equation}\label{1.28}
\sigma'(\lambda)=\frac{i}{2\pi}g'(\lambda)
+\frac{i}{2\pi}\sum_j\frac{1}{\lambda+\lambda_j}-\frac{1}{\lambda-\lambda_j}
+Q_\lambda(\lambda_j)-Q_\lambda(-\lambda_j),
\end{equation}
where
\begin{equation}
Q_\lambda(\lambda_j):=(\frac{1}{-\lambda_j})(1+(\frac{\lambda}{-\lambda_j})
+\cdots+(\frac{\lambda}{-\lambda_j})^{[m]-1}),
\end{equation}
which is also a polynomial in $\lambda$ provided
$\{\lambda_j\}\neq0$.
Furthermore, for $\lambda\in 0i+\mathbb{R}$,
\begin{equation}\label{1.29}
2\lambda {\rm
Tr}\{R(\lambda)-R_0(\lambda)-R(-\lambda)+R_0(-\lambda)\}=\sigma'(\lambda).
\end{equation}
This is the proof of Birman-Krein theorem in black box formalism
setting. See \cite{Zworski2,Zworski3}. In this case,
\begin{eqnarray}
\int_\mathbb{R}e^{i\lambda t} {\rm
Tr}U_1(t)dt=\sigma'(\lambda)+m_0(\lambda)\delta(\lambda^2)
+2\sum_{\Im\lambda_j>0}m_{\lambda_j}(R)\delta_0(\lambda^2-\lambda_j^2).\label{1.10}
\end{eqnarray}
Furthermore, when $n\geq3$, $0$ is neither a resonance nor an
eigenvalue of $R(\lambda)$. To see this, we use the asymptotic
behavior of $\sigma'(\lambda)$ near zero. From~(\ref{2.6}) by Zworski's theory \cite{Zworski3}, we can
show
\begin{equation}\label{1.31}
\sigma'(\lambda)=\lambda^{n-3} f(\lambda)\mbox{ as
}\lambda\rightarrow 0^+,\mbox{ where }f\mbox{ is smooth near
}\lambda=0.
\end{equation}
In addition, the self-adjoint operator $H$ has no $L^2$-eigenvalue
by standard argument in spectral analysis. See, \cite{Reed and
Simon}. Therefore,~(\ref{1.10}) becomes
\begin{equation}
\sigma'(\lambda)=\int_\mathbb{R}e^{i\lambda
t}{\rm Tr}U_1(t)dt.\label{2.24}
\end{equation}

\par
Using the theorem assumption, the kernel of $U_1(t,x,x)$ is
super-exponentially decaying for $|t|> T$, then~(\ref{2.24})
is in sense of Laplace transform in $\mathbb{C}$ and
\begin{equation}
\int_\mathbb{R}e^{i\lambda
t}{\rm Tr}U_1(t)dt\mbox{ converges and is entire in }\mathbb{C}.
\end{equation}
Accordingly,
\begin{lemma}
Under theorem assumption~(\ref{1.13}), the inverse Laplace transform ${\rm Tr}U_1(t)$ is unique and
 the two-phased $\sigma'(\lambda)$ has no resonance in $\mathbb{C}$.
\end{lemma}
We recall Zworski \cite[Theorem 4]{Zworski3} and the remark thereafter, for
any $\gamma>0$ and $k$,
\begin{equation}
|(\frac{\partial}{\partial
t})^k({\rm Tr}U_1(t)-\sum_{\Im\lambda\leq\gamma\log|\lambda|}m(\lambda)e^{i|t|\lambda})|\leq
C_kt^{-n+2-k},\hspace{2pt}t>t_k>\frac{n+k}{\gamma}.\label{decay}
\end{equation}
The right hand side decays polynomially when $n$ is even;
exponentially for odd $n$. Since there is no resonance, the
summation over all of the resonance is void. Consequently, letting $t\rightarrow
0$ which means $\gamma\rightarrow\infty$, we conclude that the
$\mathcal{C}^\infty$-singularity support is $\{0\}$. We can say
more.
\begin{proposition}\label{23}
Under the Theorem 1.1 assumption, the Fourier-Laplace transform
$\int_\mathbb{R}e^{i\lambda t}U_1(t,x,x)dt$ in~(\ref{2.24}) can be
extended as a Fourier-Laplace transform over $\mathbb{C}$ if $n$
is odd. In particular,~(\ref{2.24}) depends only on the short-time
behavior of $U_1(t,x,x)$:
\end{proposition}
\begin{equation}
\int_{-\infty}^{\infty}e^{i\lambda t}
U_1(t,x,x)dt=\int_{-\infty}^{\infty}e^{i\lambda t}
U_1(t,x,x)\rho_1(t)dt+\mbox{a rapidly decreasing function},
\end{equation}
for some cutoff function $\rho_1(t)$ supported at $t=0$.

\begin{proof}
We divide the Fourier transform on $U_1(t,x,y)$ into three time
intervals:
\begin{eqnarray}\nonumber
&&\langle\int_{-\infty}^\infty e^{i\lambda t}
U_1(t,x,x)dt,\varphi(\lambda)\rangle\\\nonumber
&\hspace{-4pt}:=\hspace{-4pt}&\langle\int_{-\infty}^\infty
e^{i\lambda t} U_1(t,x,x)\rho_1(t)dt+\int_{-\infty}^\infty
e^{i\lambda t} U_1(t,x,x)\rho_2(t)dt+\int_{-\infty}^\infty
e^{i\lambda
t} U_1(t,x,x)\rho_3(t)dt,\varphi(\lambda)\rangle\\
&\hspace{-4pt}:=\hspace{-4pt}&\langle
I_1+I_2+I_3,\varphi(\lambda)\rangle,\label{2.12}
\end{eqnarray}
where $\rho_i\in\mathcal{C}^\infty(\mathbb{R})$, $i=1,2,3$. Let
$\rho_1,\rho_3\in\mathcal{C}^\infty(\mathbb{R})$ be two positive
cutoff functions such that $\rho_1$ has small compact support at
$t=0$ and $\rho_3$ has support outside $(-T,T)$, for some $T>0$ as
given by~(\ref{1.13}). We also assume $\rho_1(t)$ is $1$ near
$t=0$ and $\rho_3(t)$ is $1$ near $t=\infty$. We take
$\rho_2(t)=1-\rho_1(t)-\rho_3(t)$ to be of compact support. We
take $(-T,T)\subset supp\hspace{1pt}(\rho_1+\rho_2)$. This is a
partition of unity.

Using Paley-Wiener's theorem for $I_1$,
\begin{equation}
|\int_{-\infty}^\infty e^{i\lambda t} U_1(t,x,x)\rho_1(t)dt|\leq
C(1+|\lambda|)^Ne^{h(\Im \lambda)},\label{2.30}
\end{equation}for some $N\in\mathbb{N}$ and for some constant $C$. $h$ is the support
function of $ U_1(t,x,x)\rho_1(t)$. We just keep $I_1$. $N$ will
be specified by Ivrii's result \cite{Ivrii,Ivrii2}. $I_1$ is
holomorphic in $\mathbb{C}$.
\par
We apply Paley-Wiener's theorem to $I_2$. By~(\ref{decay}), $
U_1(t,x,x)\rho_2(t)$ is a smooth function with compact support.
By construction $\rho_2(t)$ is the union of two cutoff functions.
In this case,
\begin{equation}
|I_2(\lambda)|\leq C_N(1+|\lambda|)^{-N},\forall
N\in\mathbb{N},\mbox{ whenever }\lambda\in0i+\mathbb{R}.
\end{equation}
This is a rapidly decreasing term.

\par
For $I_3$, we use the theorem assumption by letting
$N\rightarrow\infty$ in~(\ref{1.13}). We obtain
\begin{eqnarray}\nonumber
&&|I_3(\lambda)|\equiv0.
\end{eqnarray}
Therefore, the oscillatory integral
$$\int_\mathbb{R}e^{i\lambda t} U_1(t,x,y)dt=I_1(\lambda),$$
mod a rapidly decreasing term. Q.E.D.
\end{proof}
For such a short-time wave trace, we can embed the support of the
influenced set into to torus. This is finite speed of propagation.


\section{A proof on Theorem 1.1} We now prove Theorem 1.1. We
start with
\begin{eqnarray}
I_1(\lambda):=\int_{-\infty}^\infty e^{i\lambda
t}U_1(t,x,x)\rho_1(t)dt.
\end{eqnarray}
From~(\ref{2.30}), we have
\begin{equation}
|I_1(\lambda)|\leq C(1+|\lambda|)^Ne^{h|\Im \lambda|},
\end{equation}for some $N\in\mathbb{N}$ and for some constant $C$.
$h$ is the support function of $ U_1(t,x,x)\rho_1(t)$. That means
 we may choose ${\rm supp}(\rho_1(t))$ small such that
\begin{equation}
|I_1(\lambda)|\leq Ce^{\delta| \lambda|},\mbox{for some constant
}C  \mbox{ for any }\delta>0.\label{3.3}
\end{equation}
$I_1(\lambda)$ is entire function in of order $1$ of minimal type.
Moreover, we look at the local behavior of the wave trace.
According to Ivrii \cite{Ivrii2}, as
$\lambda\rightarrow0i\pm\infty$,
\begin{eqnarray}
I_1(\lambda)=\int_{-\infty}^\infty e^{i\lambda
t}U_1(t,x,x)\rho_1(t)dt\sim
a_1(x)|\lambda|^{n-2}+a_{2}(x)|\lambda|^{n-3}+\cdots,\label{3.4}
\end{eqnarray}
where $\omega_{n-1}$ is $n-1$-sphere volume and $\int_\Gamma
a_1(x)dS_x=\pm\frac{1}{4}(2\pi)^{-n+1}\omega_{n-1}{\rm
volume}(\Gamma)$. The sign depends on the boundary conditions. We
refer to \cite[Theorem 2.1]{Ivrii2} and the remark thereafter for
the Fourier transform of the short-time wave trace. We also
disregard the rapidly decreasing term from Proposition~\ref{23}.
We need a Phragem\'{e}n-Lindel\"{o}f type of lemma.
\begin{lemma}
Let $f(z)$, $z=x+iy$, be an entire function of exponential type
$\sigma$ such that $|f(x)|\leq C(1+|x|^d)$, $d\in\mathbb{N}$, then
$f(z)e^{-\sigma|y|}\leq C_1(1+|z|^d)$. In particular, when
$\sigma=0$, then $f(z)$ is a polynomial of degree no greater than
$d$.\end{lemma}
\begin{proof}
This is stated in B. Ya. Levin's book \cite[p.39]{Levin}. $\Box$
\end{proof}
Hence,
\begin{lemma}
$I_1(\lambda)$ is a polynomial of degree no greater than $n-2$.
\end{lemma}
In this case,
\begin{eqnarray}
\partial^{n-1}_\lambda I_1(\lambda)=\int_{-\infty}^\infty e^{i\lambda t}(it)^{n-1}U_1(t,x,x)\rho_1(t)dt=0,\forall\lambda\in\mathbb{C}.
\end{eqnarray}
Moreover, Fourier inversion formula tells us
\begin{equation}
t^{n-1}U_1(t,x,x)\equiv0\mbox{ in
}\mathcal{D}'((-\delta,\delta)),\forall\delta>0.\label{3.6}
\end{equation}
Since $n$ is odd,~(\ref{3.4}) and~(\ref{3.6}) implies
\begin{equation}
t^{n-1}\{\frac{a_1(x)}{t^{n-1}}+a_2(x)\delta^{(n-3)}(t)+\cdots+\mbox{
constant term}+\cdots\}\equiv0\mbox{ in
}\mathcal{D}'((-\delta,\delta)).
\end{equation}
By distribution theory at $t=0$, $$a_1(x)=0,\forall x\in\Omega.$$
In particular, $\int_\Gamma\zeta(x)a_1(x)dS_x=0$, $\forall
\zeta\in\mathcal{C}^\infty_0(\mathbb{R}^n_x;[0,1])$. In this case,
$\Gamma=\phi$. Theorem 1.1 is proved.

\end{document}